\documentclass[a4paper,12pt]{article}

\usepackage[left=2cm,right=2cm, top=2cm,bottom=2cm,bindingoffset=0cm]{geometry}

\usepackage{verbatim}
\usepackage{amsmath}
\usepackage{amsthm}
\usepackage{amssymb}
\usepackage{delarray}
\usepackage{mathrsfs}

\newcommand{\ga}{\gamma}
\newcommand{\de}{\delta}
\newcommand{\la}{\lambda}

\newcommand{\eps}{\varepsilon}

\newcommand{\iy}{\infty}
\usepackage{graphicx}
\usepackage{tikz}
\usetikzlibrary{patterns}
\usepackage{caption}
\DeclareCaptionLabelSeparator{dot}{. }
\theoremstyle{plain}

\newtheorem{thm}{Theorem}
\newtheorem{lem}{Lemma}

\newtheorem{ip}{Inverse Problem}

\theoremstyle{definition}

\newtheorem{example}{Example}

\theoremstyle{remark}

 \allowdisplaybreaks

\begin{document}

\begin{center}
{\large\bf 
Partial inverse problems for Sturm-Liouville operators on trees
}
\\[0.2cm]
{\bf Natalia Bondarenko, Chung-Tsun Shieh} \\[0.2cm]
\end{center}

\vspace{0.5cm}

{\bf Abstract.} In this paper, inverse spectral problems for Sturm-Liouville operators on a tree (a graph without cycles)
are studied. We show that if the potential on an edge is known a priori, then $b - 1$ spectral sets uniquely determine
the potential functions on a tree with $b$ external edges.
Constructive solutions, based on the method of spectral mappings, are provided for the considered inverse problems.

\medskip


\medskip


\vspace{1cm}

{\large \bf 1. Introduction.} 

\medskip

This paper concerns the theory of inverse spectral problems for Sturm-Liouville operators on geometrical graphs.
Inverse problems consist in recovering differential operators from their spectral characteristics.
Differential operators on graphs (quantum graphs) have applications in various fields of science and engineering
(mechanics, chemistry, electronics, nanoscale technology and others) and attract a considerable attention of mathematicians
in recent years. There is an extensive literature devoted to differential operators on graphs and their applications,
we mention only some research papers and surveys  
\cite{LLS94, KS97, Kuch04, Bel04, PB04, Exner08, Yur10}.

There are different kinds of inverse problems studied for quantum graphs, 
one of them is to recover the coefficients of the operator while some information is known a priori. 
This paper is focused on the reconstruction of the potential
of the Sturm-Liouville operator on a tree (a graph without cycles) with a prescribed structure and standard matching conditions
in the vertices. V.A.~Yurko \cite{Yur05, Yur06} studied such inverse problems on trees by the Weyl vector, the system of spectra
and the spectral data.
These problems are generalizations of the well--studied inverse problems for Sturm-Liouville operators
on a finite interval (see monographs \cite{Mar77, Lev84, PT87, FY01} and references therein). 
By the method of spectral mappings \cite{FY01, Yur02}, V.A.~Yurko proved uniqueness theorems and developed a constructive algorithm for solution of inverse problems
on trees.

In this paper, we formulate and solve partial inverse problems for the Sturm-Liouville operator on the tree.
We suppose that the Sturm-Liouville potential is known on the part of the graph and show that we need less data 
to recover the potential on the remaining part.
We know the only work \cite{Yang10} in this direction, where 
the potential is known on a half of one edge and
completely on the other edges of the star-shaped graph,
and the author solves the Hochstadt-Lieberman-type problem \cite{HL78} by a part of the spectrum.

In this paper, we assume that the potential is known on one edge of a tree, then reconstruct the potential on the remaining part 
by the system of spectra or the Weyl functions. 
By developing the ideas of V.A.~Yurko~\cite{Yur05, Yur06}, we show that one needs one less spectral set or one less Weyl function
for the solution of the partial inverse problem. We consider separately
the cases of boundary and internal edges, present constructive solutions and corresponding uniqueness theorems for both of them. 

The results of this paper can be generalized to the case, when the potential is known on several edges. However, in this case the number of given spectra, sufficient to recover the potential on the whole graph, depends not only on the number of these edges, but also on their location (see the example in Section 5). We note that the method of spectral mappings works also for graphs with cycles (see \cite{Yur10}), so one can generalize our results in this direction.

The paper is organized as follows.
In {\it Section~2}, we introduce the notation and briefly describe the solution of inverse problems on trees by V.A.~Yurko \cite{Yur05, Yur06}.
In {\it Section~3}, we formulate our main results and outline their constructive solutions.
{\it Section~4} contains proofs of the technical lemmas from Section~3. In {\it Section~5} we illustrate our method by an example.

\bigskip
  
{\large \bf 2. Inverse problems on a tree}

\medskip

In this section, we introduce the notation and provide the main results of V.A.~Yurko on 
the inverse problems on trees (see works \cite{Yur05, Yur06} for more details). 

Consider a compact tree $G$ with the vertices $V = \{ v_i \}_{i = 1}^{m + 1}$ and edges $E = \{ e_j \}_{j = 1}^m$.
For each vertex $v \in V$, we denote the set of edges associated with $v$ by $E_v$ and call the size of $E_v$ the {\it degree} of $v$.
Assume that the tree $G$ does not contain vertices of degree~$2$.
The vertices of degree $1$ are called {\it boundary vertices}. Denote the set of boundary vertices of the graph $G$ by $\partial G$.
For the sake of convenience, let each boundary vertex $v_i$ be an end of the edge $e_i$, such edges are called {\it boundary edges}.
All other vertices and edges are called {\it internal}. 
Let the vertex $v_r \in \partial G$ be {\it the root} of the tree.  

Each edge $e_j \in E$ is viewed as a segment $[0, T_j]$ and is parametrized by the parameter $x_j \in [0, T_j]$.
The value $x_j = 0$ correspond to one of the end vertices of the edge $e_j$, and $x_j = T_j$ corresponds to another one.
For a boundary edge, the end $x_j = 0$ corresponds to the boundary vertex $v_j$.

A function on the tree $G$ can be represented as a vector function $y = [y_j]_{j = 1}^m$,
where $y_j = y_j(x_j)$, $x_j \in [0, T_j]$, $j = \overline{1, m }$. 
Let $e_j = [v_i, v_k]$, i.e. the vertex $v_i$ corresponds to the end $x_j = 0$ and the vertex $v_k$ corresponds to 
$x_j = T_j$. Introduce the following notation
$$
\begin{array}{l}
 	y_j(v_i) = y_j(0), \quad y_j(v_k) = y_j(T_j),\\
 	y'_j(v_i) = y'_j(0), \quad y'_j(v_k) = -y'_j(T_j).
\end{array}
$$
If $v_i \in \partial G$, we omit the index of the edge and write $y(v_i)$ and $y'(v_i)$. 

Consider the Sturm-Liouville equation on $G$:
\begin{equation} \label{eqv}
-y''_j + q_j(x_j) y_j = \la y_j, \quad x_j \in [0, T_j], \quad j = \overline{1, m}.	
\end{equation}  
where $\la$ is the spectral parameter, $q_j \in L[0, T_j]$. We call the function $q = [q_j]_{j = 1}^m$ {\it the potential}
on the graph $G$. 
The functions $y_j$, $y'_j$ are absolutely continuous on the segments $[0, T_j]$ and satisfy the {\it standard matching conditions}
in the internal vertices $v \in V \backslash \partial G$:
\begin{equation} \label{MC}
	\begin{cases}
 	y_j(v) = y_k(v), \quad e_j, e_k \in E_v \quad \text{(continuity condition)}, \\
  	\sum_{e_j \in E_v} y'_j(v) = 0, \quad \text{(Kirchhoff's condition)}.
  	\end{cases}
\end{equation}

Let $L_0$ and $L_k$, $v_k \in \partial G$, be the boundary value problem for system \eqref{eqv} with the matching conditions \eqref{MC} 
and the following conditions in the boundary vertices:
\begin{align} \label{boundL}
 	L_0 \colon & y(v_i) = 0, \quad v_i \in \partial G, \\
 	\label{boundLk}
 	L_k \colon & y'(v_k) = 0, \quad y(v_i) = 0, \quad v_i \in \partial G \backslash \{ v_k \}.
\end{align}
It is well-known, that the problems $L_k$ have discrete spectra, which are 
the countable sets of eigenvalues 
$\Lambda_k = \{ \la_{ks} \}_{s = 1}^{\iy}$, $k = 0$ or $v_k \in \partial G$.

Fix a boundary  vertex $v_k \in \partial G$. Let $\Psi_k = [\psi_{kj}]_{j = 1}^m$, $\psi_{kj} = \psi_{kj}(x_j, \la)$,
be the solution of the system \eqref{eqv}, satisfying the matching conditions \eqref{MC} and 
the boundary conditions
$$
  	\psi_{kk}(0, \la) = 1, \quad \psi_{kj}(0, \la) = 0, \quad v_j \in \partial G \backslash \{ v_k \}. 
$$
Denote $M_k(\la) = \psi'_{kk}(0, \la)$. 
The functions $\Psi_k$ and $M_k$ are called {\it the Weyl solution} and {\it the Weyl function} of \eqref{eqv} with respect to the boundary vertex $v_k$,
respectively. The notion of the Weyl function for the tree generalizes the notion of the Weyl function
($m$-function) for the classical Sturm�-Liouville operator on a finite interval \cite{Mar77, FY01}. 
If the tree $G$ consists of only one edge, then $M_k(\la)$ coincide with the classical Weyl function.

Consider the following inverse problems.

\begin{ip} \label{ip:1}
Given the spectra $\Lambda_0$, $\Lambda_k$, $v_k \in \partial G \backslash \{ v_r \}$, construct the potential $q$
on the tree $G$.
\end{ip}

\begin{ip} \label{ip:2}
Given the Weyl functions $M_k(\la)$, $v_k \in \partial G \backslash \{ v_r \}$, construct the potential $q$
on the tree $G$.
\end{ip}

Note that if the number of boundary vertices is $b$, then one needs $b$ spectra or $b - 1$ Weyl functions to recover the 
potential. We do not require the data associated with the root $v_r$.

There is a close relation between Inverse Problems~\ref{ip:1} and \ref{ip:2}. The Weyl functions can be represented in the form
\begin{equation} \label{reprMk}
M_k(\la) = -\frac{\Delta_k(\la)}{\Delta_0(\la)}, \quad v_k \in \partial G,
\end{equation}
where $\Delta_k(\la)$ are characteristic functions of the boundary value problems $L_k$.
If the eigenvalues $\Lambda_k$ are known, one can construct characteristic functions as infinite products by Hadamard's theorem.
Thus, with the system of spectra, one can obtain the Weyl functions and reduce Inverse Problem~\ref{ip:1} to 
Inverse Problem~\ref{ip:2}.

V.A. Yurko has proved, that Inverse Problems~\ref{ip:1} and \ref{ip:2} are uniquely solvable, and provided a constructive algorithm
for the solution by the method of spectral mappings \cite{FY01}. 
In the remaining of this section, we shall briefly describe his algorithm. Let the Weyl functions $M_k(\la)$, $v_k \in \partial G \backslash \{ v_r \}$ be given. 
Consider the following auxiliary problem.

\medskip

{\bf Problem IP(k).} Given $M_k(\la)$, construct the potential $q_k(x_k)$ on the edge $e_k$. 

\medskip

Note that this problem is not equivalent to the inverse problem on the finite interval, since the Weyl function
$M_k(\la)$ contains information from the whole graph.
However, it can be solved uniquely by the method of spectral mappings, and the potential on the boundary edges can be recovered.
Then V.A.~Yurko used so-called $\mu$-procedure to recover the potential on the internal edges. We reformulate these ideas
in the form, which is more convenient for us in the future.

\begin{thm} \label{thm:cut}
Let $v$ be an internal vertex, connected with the set of boundary vertices $V' \subset \partial G \backslash \{ v_r \}$ and only one other vertex.
Suppose the potentials $q_k$ on the edges $e_k$ are known for all $v_k \in V'$, as well as a Weyl function $M_k(\la)$
for at least one vertex from the set $V'$. Denote $G'$ the graph by removing the vertices $v_k \in V'$ together with 
the corresponding edges $e_k$ from the graph $G$.
Then the Weyl function for the graph $G'$ with respect the the vertex $v$ can be determined from the given information.
\end{thm}

Applying Theorem~\ref{thm:cut}, one can cut the boundary edges off, until the potential will be recovered on the whole graph.

\bigskip

{\large \bf 3. Partial inverse problems}

\medskip

In this section, the main results of the paper are formulated.
We assume that the potential is known on one edge of the tree
and formulate partial inverse problems.
We consider separately the cases of boundary and internal edge.
The first one appears to be trivial, for the second one we describe the procedure 
of the constructive solution. For the convenience of the reader, 
the proofs of the technical lemmas are provided  in Section~4.

\begin{ip} \label{ip:3}
Let $e_f$ be a boundary edge ($f \ne r$). Given the potential $q_f$ on the edge $e_f$
and the spectra $\Lambda_0$, $\Lambda_k$, $v_k \in \partial G \backslash \{ v_f, v_r \}$.
Construct the potential $q$ on the tree $G$.
\end{ip}

The solution of Inverse Problem~3 is a slight modification of the method described in Section~2. 
From $\Lambda_0$, $\Lambda_k$, $v_k \in \partial G \backslash \{ v_f, v_r \}$, we easily construct the potentials $q_k$ for $v_k \in \partial G \backslash \{ v_f, v_r \}$.
The potential $q_f$ is known, so we can apply Theorem~\ref{thm:cut} iteratively
and recover the potential on $G$.

Now let $e_f$ be an internal edge. If this edge is removed, the graph splits into two parts, call them $P_1$ and $P_2$.
Let $\partial P_1$ and $\partial P_2$ be the sets of boundary vertices of $P_1$ and $P_2$, respectively. 
Fix two arbitrary vertices $v_{r1} \in \partial P_1$ and $v_{r2} \in \partial P_2$. 

\begin{ip} \label{ip:main}
Given the potential $q_f$ on the internal edge $e_f$, the spectra $\Lambda_0$, $\Lambda_k$
$v_k \in \partial G \backslash \{ v_{r1}, v_{r2} \}$. Construct the potential $q$ on the tree $G$.
\end{ip}

%Below we outline of the solution of Inverse Problem~\ref{ip:main}.
{\bf Solution of Inverse Problem~\ref{ip:main}}. For simplicity, we assume that the ends of the edge $e_f$ 
have degree $3$. The general case requires minor modifications. If one splits each of the ends of $e_f$ into three vertices,
the tree splits into five subtrees $G_i$, $i = \overline{1, 5}$, such that $v_{r1} \in G_2$, $v_{r2} \in G_5$, and $G_3$ contains
the only edge $e_f$ (see fig.~\ref{img:1}). Let $v_1$ and $v_4$ are arbitrary boundary vertices of the trees $G_1$ and $G_4$ (different from the ends of $e_f$),
$v_{r1} = v_2$, $v_{r2} = v_5$, $e_f = [v_3, v_6]$. 

\begin{figure}[h!]
\centering
\begin{tikzpicture}
\filldraw (5, 4) circle (2pt) node[anchor=west]{$v_6$};
\filldraw (5, 5) circle (2pt) node[anchor=west]{$v_3$};
% G1
\draw[dashed] (2.7, 8.45) edge [bend left] (5, 5);
\draw[dashed] (1, 6) edge [bend right] (5, 5);
\draw[dashed] (2.7, 8.45) arc (60:230:1.5);
\filldraw (1, 8) circle (2pt) node[anchor=north]{$v_1$};
\draw (5, 5) edge (4, 6);
\draw (4, 6) edge (3.5, 7);
\draw (4, 6) edge (2, 7);
\draw (2, 7) edge (1, 8);
\draw (4, 6) edge (3, 6);
\draw (3, 6) edge (2, 6.5);
\draw (3, 6) edge (2, 5.5);
\draw (2, 7) edge (2, 8.3);
\draw (2, 7) edge (0.7, 7);
\draw (2.5, 8) node{$G_1$};
\filldraw (4, 6) circle (1pt);
\filldraw (3.5, 7) circle (1pt);
\filldraw (2, 7) circle (1pt);
\filldraw (3, 6) circle (1pt);
\filldraw (2, 6.5) circle (1pt);
\filldraw (2, 5.5) circle (1pt);
\filldraw (2, 8.3) circle (1pt);
\filldraw (0.7, 7) circle (1pt);
% G2
\draw[dashed] (7.3, 8.45) edge [bend right] (5, 5);
\draw[dashed] (9, 6) edge [bend left] (5, 5);
\draw[dashed] (9, 6) arc (-50:120:1.5);
\filldraw (9, 8) circle (2pt) node[anchor=north]{$v_2$};
\draw (5, 5) edge (6, 6);
\draw (6, 6) edge (6.5, 7);
\draw (6, 6) edge (8, 7);
\draw (8, 7) edge (9, 8);
\draw (6, 6) edge (7, 6);
\draw (7, 6) edge (8, 6.5);
\draw (7, 6) edge (8, 5.5);
\draw (8, 7) edge (8, 8.3);
\draw (8, 7) edge (9.3, 7);
\draw (7.5, 8) node{$G_2$};
\filldraw (6, 6) circle (1pt);
\filldraw (6.5, 7) circle (1pt);
\filldraw (8, 7) circle (1pt);
\filldraw (7, 6) circle (1pt);
\filldraw (8, 6.5) circle (1pt);
\filldraw (8, 5.5) circle (1pt);
\filldraw (8, 8.3) circle (1pt);
\filldraw (9.3, 7) circle (1pt);
% G3
\draw[thick] (5, 4) edge node[auto]{$G_3$} (5, 5);
% G4
\draw[dashed] (2.7, 0.55) edge [bend right] (5, 4);
\draw[dashed] (1, 3) edge [bend left] (5, 4);
\draw[dashed] (1, 3) arc (130:300:1.5);
\filldraw (1, 1) circle (2pt) node[anchor=west]{$v_4$};
\draw (5, 4) edge (4, 3);
\draw (4, 3) edge (3.5, 2);
\draw (4, 3) edge (2, 2);
\draw (2, 2) edge (1, 1);
\draw (4, 3) edge (3, 3);
\draw (3, 3) edge (2, 2.5);
\draw (3, 3) edge (2, 3.5);
\draw (2, 2) edge (2, 0.7);
\draw (2, 2) edge (0.7, 2);
\draw (2.5, 1) node{$G_4$};
\filldraw (4, 3) circle (1pt);
\filldraw (3.5, 2) circle (1pt);
\filldraw (2, 2) circle (1pt);
\filldraw (3, 3) circle (1pt);
\filldraw (2, 2.5) circle (1pt);
\filldraw (2, 3.5) circle (1pt);
\filldraw (2, 0.7) circle (1pt);
\filldraw (0.7, 2) circle (1pt);
% G5
\draw[dashed] (9, 3) edge [bend right] (5, 4);
\draw[dashed] (7.3, 0.55) arc (-120:50:1.5);
\draw[dashed] (7.3, 0.55) edge [bend left] (5, 4);
\filldraw (9, 1) circle (2pt) node[anchor=east]{$v_5$};
\draw (5, 4) edge (6, 3);
\draw (6, 3) edge (6.5, 2);
\draw (6, 3) edge (8, 2);
\draw (8, 2) edge (9, 1);
\draw (6, 3) edge (7, 3);
\draw (7, 3) edge (8, 2.5);
\draw (7, 3) edge (8, 3.5);
\draw (8, 2) edge (8, 0.7);
\draw (8, 2) edge (9.3, 2);
\draw (7.5, 1) node{$G_5$};
\filldraw (6, 3) circle (1pt);
\filldraw (6.5, 2) circle (1pt);
\filldraw (8, 2) circle (1pt);
\filldraw (7, 3) circle (1pt);
\filldraw (8, 2.5) circle (1pt);
\filldraw (8, 3.5) circle (1pt);
\filldraw (8, 0.7) circle (1pt);
\filldraw (9.3, 2) circle (1pt);
\end{tikzpicture}
\caption{}
\label{img:1}
\end{figure}

\smallskip

{\bf Step 1.} Construct the characteristic functions $\Delta_k(\la)$ by the given spectra $\Lambda_k$, 
$k = 0$ and $v_k \in \partial G \backslash \{ v_2, v_5 \}$. Find $M_k(\la)$ by formula \eqref{reprMk}.

\smallskip

{\bf Step 2.} Consider trees $G_1$ and $G_4$. Recover the potential $q$ on the
edges of $G_1$ and $G_4$, using the solutions of the problems IP(k) for $v_k \in \partial G_1 \backslash \{ v_3 \}$
and $v_k \in \partial G_4 \backslash \{ v_6 \}$, and them applying Theorem~\ref{thm:cut} iteratively.

\smallskip

{\bf Step 3.} Introduce the characteristic functions of the boundary value problems for the Sturm-Liouville equations
\eqref{eqv} on the graphs $G_1$-$G_5$ with the standard matching conditions \eqref{MC} in internal vertices
and the following conditions in the boundary vertices:
\begin{align*}
 	\text{graph} \, G_1 \quad &
 	\left\{
 	\begin{array}{ll}
 	  \Delta^{DD}_1(\la) \colon & y(v_k) = 0, \quad v_k \in \partial G_1, \\
 	  \Delta^{ND}_1(\la) \colon & y'(v_1) = 0, \quad y(v_k) = 0, \quad v_k \in \partial G_1 \backslash \{ v_1 \}, \\
 	  \Delta^{DN}_1(\la) \colon & y'(v_3) = 0, \quad y(v_k) = 0, \quad v_k \in \partial G_1 \backslash \{ v_3 \}, \\
 	  \Delta^{NN}_1(\la) \colon & y'(v_1) = 0, \quad y'(v_3) = 0, \quad y(v_k) = 0, \quad v_k \in \partial G_1 \backslash \{ v_1, v_3 \}.  
 	\end{array} \right.       \\
 	\text{graph} \, G_2 \quad &
 	\left\{
 	\begin{array}{ll}
 	  \Delta^{D}_2(\la) \colon & y(v_k) = 0, \quad v_k \in \partial G_2, \\
 	  \Delta^{N}_2(\la) \colon & y'(v_3) = 0, \quad y(v_k) = 0, \quad v_k \in \partial G_2 \backslash \{ v_3 \}.
 	\end{array} \right. \\
 	\text{graph} \, G_3 \quad &
 	\left\{ \begin{array}{ll}
 	  \Delta^{DD}_3(\la) \colon & y(v_3) = 0, \quad y(v_6) = 0, \\
 	  \Delta^{ND}_3(\la) \colon & y'(v_3) = 0, \quad y(v_6) = 0,\\
 	  \Delta^{DN}_3(\la) \colon & y(v_3) = 0, \quad y'(v_6) = 0, \\
 	  \Delta^{NN}_3(\la) \colon & y'(v_3) = 0, \quad y'(v_6) = 0.  
 	\end{array} \right.       \\
 	\text{graph} \, G_4 \quad &
 	\left\{ \begin{array}{ll}
 	  \Delta^{DD}_4(\la) \colon & y(v_k) = 0, \quad v_k \in \partial G_4, \\
 	  \Delta^{ND}_4(\la) \colon & y'(v_4) = 0, \quad y(v_k) = 0, \quad v_k \in \partial G_4 \backslash \{ v_4 \}, \\
 	  \Delta^{DN}_4(\la) \colon & y'(v_6) = 0, \quad y(v_k) = 0, \quad v_k \in \partial G_4 \backslash \{ v_6 \}, \\
 	  \Delta^{NN}_4(\la) \colon & y'(v_4) = 0, \quad y'(v_3) = 0, \quad y(v_k) = 0, \quad v_k \in \partial G_4 \backslash \{ v_4, v_6 \}.  
 	\end{array} \right.      \\
 	\text{graph} \, G_5 \quad &
 	\left\{ \begin{array}{ll}
 	  \Delta^{D}_5(\la) \colon & y(v_k) = 0, \quad v_k \in \partial G_5, \\
 	  \Delta^{N}_5(\la) \colon & y'(v_6) = 0, \quad y(v_k) = 0, \quad v_k \in \partial G_5 \backslash \{ v_6 \}.
 	\end{array} \right.	
\end{align*}

\begin{lem} \label{lem:det}
The following relation holds
\begin{equation} \label{relDelta}
 	\Delta_0(\la) = \begin{vmatrix}
 	                 \Delta^{DD}_1(\la) & -\Delta^D_2(\la) & 0 & 0 & 0 & 0 \\
 	                 0 & \Delta^D_2(\la) & -1 & 0 & 0 & 0 \\
 	                 \Delta^{DN}_1(\la) & \Delta^N_2(\la) & 0 & -1 & 0 & 0 \\
 	                 0 & 0 & \Delta^{ND}_3(\la) & \Delta^{DD}_3(\la) & -\Delta^{DD}_4(\la) & 0 \\
 	                 0 & 0 & 0 & 0 & \Delta^{DD}_4(\la) & -\Delta^D_5(\la) \\
 	                 0 & 0 & \Delta^{NN}_3(\la) & \Delta^{DN}_3(\la) & \Delta^{DN}_4(\la) & \Delta^{N}_5(\la)
 	              \end{vmatrix}.
\end{equation}
If one changes $\Delta^{DD}_1(\la)$ to $\Delta^{ND}_1(\la)$ and $\Delta^{DN}_1(\la)$ to $\Delta^{NN}_1(\la)$, he obtains the determinant equal
to $\Delta_1(\la)$. Similarly, if one changes $\Delta^{DD}_4(\la)$ to $\Delta^{ND}_4(\la)$ and $\Delta^{DN}_4(\la)$ to $\Delta^{NN}_4(\la)$,
he gets $\Delta_4(\la)$. 
\end{lem}

{\bf Step 4.} Note that the functions $\Delta_0(\la)$, $\Delta_1(\la)$, $\Delta_4(\la)$ are known from Step~1.
Since we know the potential on the graphs $G_1$, $G_4$ (from Step~2) and $G_3$ (given a priori), we can easily construct 
the characteristic functions for these graphs. Consider the relation \eqref{relDelta} and similar relations for $\Delta_1(\la)$ and $\Delta_4(\la)$
as a system of equations with respect to $\Delta^D_2(\la)$, $\Delta^N_2(\la)$, $\Delta^D_5(\la)$ and $\Delta^N_5(\la)$ in the following form
\begin{equation} \label{systemA}
\begin{cases}
a_{11} \Delta^D_2 \Delta^D_5 + a_{12} \Delta^N_2 \Delta^D_5 + a_{13} \Delta^D_2 \Delta^N_5 + a_{14} \Delta^N_2 \Delta^N_5 = \Delta_0, \\
a_{21} \Delta^D_2 \Delta^D_5 + a_{22} \Delta^N_2 \Delta^D_5 + a_{23} \Delta^D_2 \Delta^N_5 + a_{24} \Delta^N_2 \Delta^N_5 = \Delta_1, \\
a_{31} \Delta^D_2 \Delta^D_5 + a_{32} \Delta^N_2 \Delta^D_5 + a_{33} \Delta^D_2 \Delta^N_5 + a_{34} \Delta^N_2 \Delta^N_5 = \Delta_4,
\end{cases}
\end{equation}
where $a_{ij} = a_{ij}(\la)$, $i = \overline{1, 3}$, $j = \overline{1, 4}$, are known coefficients.

\smallskip

{\bf Step 5.} Multiply the first equation of \eqref{systemA} by $\Delta_1$ and subtract the second equations, multiplyed by $\Delta_0$.
Apply the similar trasform to the first and the third equations. Then we obtain the system
$$
\begin{cases}
b_{11} \Delta^D_2 \Delta^D_5 + b_{12} \Delta^N_2 \Delta^D_5 + b_{13} \Delta^D_2 \Delta^N_5 + b_{14} \Delta^N_2 \Delta^N_5 = 0, \\
b_{21} \Delta^D_2 \Delta^D_5 + b_{22} \Delta^N_2 \Delta^D_5 + b_{23} \Delta^D_2 \Delta^N_5 + b_{24} \Delta^N_2 \Delta^N_5 = 0, \\
\end{cases}		
$$ 
where 
\begin{equation} \label{defb}
	b_{1i} = a_{1i} \Delta_1 - a_{2 i} \Delta_0, \quad b_{2i} = a_{1i} \Delta_4 - a_{3 i} \Delta_0, \quad i = \overline{1, 4}.
\end{equation}
Divide both equations by $\Delta^D_2 \Delta^D_5$.
\begin{equation} \label{systemB}
   b_{i1} + b_{i2} \tilde M_2 + b_{i3} \tilde M_5 + b_{i4} \tilde M_2 \tilde M_5 = 0, \quad i= 1, 2, 
\end{equation}
where 
$$
  	\tilde M_2(\la) = \frac{\Delta^N_2(\la)}{\Delta^D_2(\la)}, \quad \tilde M_5(\la) = \frac{\Delta^N_5(\la)}{\Delta^D_5(\la)}
$$
are (up to the sign) the Weyl functions for the subtrees $G_2$ and $G_5$ associated with the vertices $v_3$ and $v_6$, respectively. 

\smallskip

{\bf Step 6.} From the system \eqref{systemB} we easily derive
$$
 	\tilde M_5 = -\frac{b_{i1} + b_{i2} \tilde M_2}{b_{i3} + b_{i4} \tilde M_2}, \quad i = 1, 2.
$$
Hence
$$
 	(b_{11} + b_{12} \tilde M_2) (b_{23} + b_{24} \tilde M_2) = (b_{21} + b_{22} \tilde M_2) (b_{13} + b_{14} \tilde M_2).
$$
Finally, we obtain the quadratic equation with respect to $\tilde M_2(\la)$:
\begin{equation} \label{quad}
 	A(\la) \tilde M_2^2(\la) + B(\la) \tilde M_2(\la) + C(\la) = 0,
\end{equation}
with analytic coefficients $A(\la)$, $B(\la)$, $C(\la)$:
\begin{equation}  \label{defABC}
\begin{array}{l}
 	A = b_{12} b_{24} - b_{22} b_{14}, \\
  	B = b_{11} b_{24} + b_{12} b_{23} - b_{21} b_{14} - b_{22} b_{13}, \\
  	C = b_{11} b_{23} - b_{21} b_{13}.
\end{array}
\end{equation}

{\bf Step 7.}
Consider the Sturm-Liouville equation \eqref{eqv} on the tree $G$ with the potential $q = 0$.
Implement Steps 1--6 for this case and obtain the quadratic equation
\begin{equation} \label{quad0}
 	A_0(\la) \tilde M_{20}^2(\la) + B_0(\la) \tilde M_{20}(\la) + C_0(\la) = 0,
\end{equation}
analogous to \eqref{quad}.
Denote $\rho = \sqrt \la$, $\mbox{Re}\, \rho \ge 0$, $S_{\de} := \{ \rho \colon \mbox{Re}\, \rho \ge 0, \quad | \mbox{Im}\, \rho | \le \de \}$, $\de > 0$,
$[1] = 1 + O(\rho^{-1})$. Let $f(\rho^2)$ be an analytic function and $\eps > 0$. Denote
$Z_{\eps}(f) := \{ \rho \colon |f(\rho^2)| \ge \eps\}$.

\begin{lem} \label{lem:asymptABC}
The following asymptotic relations hold 
$$
 	A(\la) = A_0(\la)[1], \quad B(\la) = B_0(\la)[1], \quad C(\la) = C_0(\la)[1], \quad \rho \in S_{\de} \cap Z_{\eps}(A_0 B_0 C_0), \, |\rho| \to \iy. 
$$
\end{lem}

Consequently, $D(\la) = D_0(\la)[1]$ for $\rho \in S_{\de} \cap Z_{\eps}(D_0), \, |\rho| \to \iy$, where $D(\la)$ and $D_0(\la)$ 
are discriminants of equations \eqref{quad} and \eqref{quad0}, respectively.

\begin{lem} \label{lem:discr}
$A_0(\la) \not \equiv 0$, $D_0(\la) \not \equiv 0$.
\end{lem}

It follows from Lemmas~\ref{lem:asymptABC} and \ref{lem:discr}, that the quadratic equation \eqref{quad}
does not degenerate for $\rho \in S_{\de} \cap Z_{\eps}(A_0 D_0)$, and two roots of \eqref{quad} are different by 
asymptotics as $|\rho| \to \iy$. 
One can easily find an asymptotic representation of $\tilde M_2(\la)$ for any particular graph and choose the correct root
of \eqref{quad} on some region of $S_{\de}$ for sufficiently large $|\rho|$. Then the function $\tilde M_2(\la)$ can be constructed 
for all $\la \in \mathbb{C}$ except its singularities by analytic continuation. Similarly one can find $\tilde M_5(\la)$.

\smallskip

{\bf Step 8.}
Consider the tree $G_2$ with the root $v_2$. Solve problems IP(k) by $M_k(\la)$, $v_k \in \partial G_2 \backslash \{ v_2, v_6 \}$,
and by $\tilde M_2(\la)$ for $v_3$, obtain the potential on the boundary edges except $e_2$. Then apply the cutting 
of boundary edges by Theorem~\ref{thm:cut} and recover the potential $q$ on $G_2$. The subtree $G_5$ can be treated similarly.

Thus, we recovered the potential $q$ on the whole graph $G$. In parallel, we have proved the following uniqueness theorem.

\begin{thm}
Let the potential $q_f$ on the edge $e_f$ ($f \ne r$) be known. 

(i) If $e_f$ in a boundary edge, the spectra $\Lambda_0$, $\Lambda_k$, $v_k \in \partial G \backslash \{ v_f, v_r \}$, 
uniquely determine the potential $q$ on the whole graph $G$. 

(ii) If $e_f$ is an internal edge, the spectra $\Lambda_0$, $\Lambda_k$
$v_k \in \partial G \backslash \{ v_{r1}, v_{r2} \}$ uniquely determine the potential $q$ on the whole graph $G$.

\end{thm}

Using the described methods with some technical modifications, one can solve partial inverse problems by Weyl functions.

\begin{ip}
Let $e_f$ be a boundary edge ($f \ne r$). Given the potential $q_f$ on the edge $e_f$
and the Weyl functions $M_k(\la)$, $v_k \in \partial G \backslash \{ v_f, v_r \}$.
Construct the potential $q$ on the tree $G$.
\end{ip}

\begin{ip} 
Given the potential $q_f$ on the internal edge $e_f$, the Weyl functions $M_k(\la)$
$v_k \in \partial G \backslash \{ v_{r1}, v_{r2} \}$. Construct the potential $q$ on the tree $G$.
\end{ip}

Thus, if the number of boundary edges is $b$ and the potential is known on one edge (boundary or internal),
$b - 2$ Weyl functions are required to construct $q$ on the whole graph.

\bigskip

{\large \bf 4. Proofs}

\medskip

{\bf 4.1. Proof of Lemma~\ref{lem:det}.}
Consider the Sturm-Liouville equation \eqref{eqv} on the tree $G$. 
Let $C_j(x_j, \la)$ and $S_j(x_j, \la)$ be solutions of \eqref{eqv} on the edge $e_j$
under initial conditions
$$
 	C_j(0, \la) = S'_j(0, \la) = 1, \quad C'_j(0, \la) = S_j(0, \la) = 0.
$$
Any solution $y = [y_j]_{j = 1}^m$ of the equation \eqref{eqv} on $G$ admits the following representation
\begin{equation} \label{expy}
 	y_j(x_j, \la) = M_j^0(\la) C_j(x_j, \la) + M_j^1(\la) S_j(x_j, \la), \quad j = \overline{1, m}, \, x_j \in [0, T_j].
\end{equation}

Let $BC$ be some fixed boundary conditions in the vertices $v \in \partial G$ of the form $y(v) = 0$ or $y'(v) = 0$
(for instance, we consider conditions \eqref{boundL} for the problem $L$ and \eqref{boundLk} for the problem $L_k$).
Denote by $L$ the boundary value problem for the Sturm-Liouville equation \eqref{eqv} with the standard matching conditions \eqref{MC}
and the boundary conditions $BC$.
If $y$ is a solution of a boundary value problem $L$, substitute \eqref{expy}
into \eqref{MC} and $BC$, and obtain
a linear algebraic system with respect to $M_j^0(\la)$, $M_j^1(\la)$. It is easy to check that
the determinant of this system is a characteristic function $\Delta(\la)$ of the boundary value problem $L$,
i.e. zeros of $\Delta(\la)$ coincide with the eigenvalues of $L$. 

\begin{example}
Consider the problem $L_0$ for the star-type graph for $m = 3$. Then boundary conditions \eqref{boundL} yield
$M_1^0(\la) = M_2^0(\la) = M_3^0(\la) = 0$. Consequently, from \eqref{MC} we obtain the system with respect to $M_j^0(\la)$, $j = 1, 2, 3$,
with the determinant
$$
 	\Delta_0(\la) = \begin{vmatrix}
 	S_1(T_1, \la) & -S_2(T_2, \la) & 0 \\
 	0 & S_2(T_2, \la) & -S_3(T_3, \la) \\
 	S'_1(T_1, \la) & S'_2(T_2, \la) & S'_3(T_3, \la)
 	\end{vmatrix}.
$$ 
\end{example}

In the general case, the following assertion is valid.

\begin{lem} \label{lem:detg}
Let $w \in V$ and the degree of $w$ be equal $n$. Splitting the vertex $w$, we split $G$ into $n$ subtrees
$G_i$, $i = \overline{1, n}$. 
For each $i = \overline{1, n}$, let $\Delta^D_i(\la)$ and $\Delta^N_i(\la)$ be characteristic functions
for boundary value problems for equation \eqref{eqv} on tree $G_i$ with matching conditions \eqref{MC},
boundary conditions $BC$ for $v \in \partial G \cap \partial G_i$ and
the Dirichlet condition $y(u) = 0$ for $\Delta^D_i(\la)$ and the Neumann condition $y'(u) = 0$ for $\Delta^N_i(\la)$.
Then the characteristic function $\Delta(\la)$ for $G$ with the conditions \eqref{MC} and $BC$ 
admits the following representation:
\begin{equation} \label{detg}
 	\Delta(\la) = \begin{vmatrix}
 	               \Delta^D_1(\la) & -\Delta^D_2(\la) & 0 & \dots & 0 \\
 	               0 & \Delta^D_2(\la) & -\Delta^D_3(\la) & \dots & 0 \\
 	               \hdotsfor{5} \\
 	               0 & 0 & 0 & \dots -\Delta^D_n(\la) \\
					\Delta^N_1(\la) & \Delta^N_2(\la) & \Delta^N_3(\la) & \dots & \Delta^N_n(\la)
 				 \end{vmatrix}.	
\end{equation}
\end{lem}

Indeed, if we write the determinant for $\Delta(\la)$ and analyze the participation of the edges 
of $G_i$ in this determinant, we can easily see that 
$\Delta(\la) = \Delta_i^D(\la) D_i(\la) + \Delta_i^N(\la) E_i(\la)$, where the functions $D_i(\la)$ and $E_i(\la)$
do not depend on the subtree $G_i$. Thus we can consider the simplest case of the star-type graph, when each $G_i$ contains only one edge,
and then change the multipliers, corresponding to subgraphs $G_i$, to $\Delta_i^D(\la)$ and $\Delta_i^N(\la)$. 
Thus we directly obtain \eqref{detg} from the formula for the star-type graph.

Lemma~\ref{lem:det} follows from Lemma~\ref{detg} for the graph in the fig.~\ref{img:1}.
Alternatively, one can derive \eqref{relDelta} from \eqref{systemex}, changing characteristic functions for one-edge subtrees by
general characteristic function.

\medskip

{\bf 4.2. Proof of Lemma~\ref{lem:asymptABC}.}
Together with $L$ consider the boundary value problem $L^0$ for equation \eqref{eqv} with $q \equiv 0$, the matching conditions \eqref{MC}
and the boundary conditions $BC$. If some symbol $\ga$ denotes the object related to $L$, we denote by the symbol $\ga^0$ the
similar object related to $L^0$. In particular, $\Delta^0(\la)$ is the characteristic function of $L^0$.
Let the symbol $P(\rho)$ stand for different polynomials of $\sin \rho T_j$ and $\cos \rho T_j$, $j = \overline{1, m}$.

\begin{lem} \label{lem:asymptDelta}
The characteristic function $\Delta(\la)$ has the following asymptotic behavior:
$$
 	\Delta(\la) = \Delta^0(\la) + O(\rho^{-d}) = \frac{P(\rho)}{\rho^{d-1}} + O(\rho^{-d}), \quad \rho \in S_{\de}, \, |\rho| \to \iy,
$$
where $P(\rho) \not \equiv 0$
and $d = m - i - n$, where $m$ is the number of the edges, $i$ is the number of internal vertices and $n$ is the number of boundary vertices with 
the Neumann boundary condition $y'(v) = 0$.
\end{lem}

\begin{proof}
The claim of the lemma immediately follows from the standard asymptotic formulas
$$
 	C_j(x_j, \la) = \cos \rho x_j + O(\rho^{-1}), \quad C'_j(x_j, \la) = -\rho \sin \rho x_j + O(1),
$$
$$
 	S_j(x_j, \la) = \frac{\sin \rho x_j}{\rho} + O(\rho^{-2}), \quad S'_j(x, \la) = \cos \rho x_j, \quad \rho \in S_{\de}, \, |\rho|\to \iy,
$$
and the construction of $\Delta(\la)$. The relation $P(\rho) \not \equiv 0$ follows from the regularity of the standard matching conditions.
\end{proof}

Applying Lemma~\ref{lem:asymptDelta} to the characteristic functions, defined on Step~3 of the algorithm, we derive
asymptotic representations for the coefficients $c = a_{ij}, b_{ij}, A, B, C$ in the following form:
$$
 	c(\la) = c^0(\la) + O(\rho^{-d}) = \frac{P(\rho)}{\rho^{d-1}} + O(\rho^{-d}), \quad \rho \in S_{\de}, \, |\rho| \to \iy,
$$
where $d$ stands for different integers. This relation yields Lemma~\ref{lem:asymptABC}.

\medskip

{\bf 4.3. Proof of Lemma~\ref{lem:discr}.}
In this subsection, we consider only the problem $L^0$ with $q \equiv 0$, so we omit the index $0$ for brevity.
For simplicity, let $T_f = 1$.
Taking into account, that
$$
	\Delta_3^{DD} = \frac{\sin \rho}{\rho}, \quad \Delta_3^{ND} = \Delta_3^{DN} = \cos \rho, \quad
	\Delta_3^{NN} = -\rho \sin \rho
$$
and doing some algebra with the expressions \eqref{relDelta}, \eqref{defb}, \eqref{defABC}, we derive
\begin{align} \label{formA}
 	A(\la) & = -F_1(\la) F_4(\la) \Delta_0(\la) \frac{\sin^2\rho}{\rho^2} \Delta_4^{DD}(\la) \Delta_5^D(\la) \chi(\la), \\ \label{formB}
  	B(\la) & = -F_1(\la) F_4(\la) \frac{\sin \rho}{\rho} \Delta_0(\la) \left\{ \Delta_5^D(\la) \Pi(\la) + 
  	\Delta_5^D(\la) \frac{\sin \rho}{\rho} \xi(\la) - \Delta_4^{DD}(\la) \Delta_5^N(\la) \frac{\sin \rho}{\rho} \chi(\la)  \right\}, \\ \label{formC}
  	C(\la) & = F_1(\la) F_4(\la) \Delta_0(\la) \frac{\sin \rho}{\rho} \Delta_5^N(\la) \left\{ \Pi(\la) + \frac{\sin \rho}{\rho} \xi(\la) \right\},
\end{align}
where
$$
 	F_i(\la) = \Delta^{DD}_i(\la) \Delta^{NN}_i(\la) - \Delta^{DN}_i(\la) \Delta^{ND}_i(\la), \quad i = 1, 4, 
$$
$$
 	\Pi(\la) = 2 \Delta_1^{DD}(\la) \Delta_2^{DD}(\la) \Delta_4^{DD}(\la), 
$$
$\chi(\la)$ and $\xi(\la)$ are characteristic functions of the graphs $G_1 \cup G_2 \cup G_3$ and 
$G_1 \cup G_2 \cup G_3 \cup G_4$, respectively. Here we mean that the copies of the vertex $v_3$ (and $v_6$ in the second graph)
are joined into one vertex with the standard matching conditions \eqref{MC}.

\begin{lem} \label{lem:ND}
Let $v_1$ and $v_2$ be two fixed vertices from $\partial G$. Denote by $\Delta^{DD}(\la)$, $\Delta^{DN}(\la)$,
$\Delta^{ND}(\la)$ and $\Delta^{NN}(\la)$ the characteristic functions for equation \eqref{eqv} on the tree $G$
with the matching conditions \eqref{MC}, with the following boundary conditions:
$$
    \begin{array}{ll}
       \Delta^{DD}(\la) \colon & \quad y(v_1) = y(v_2) = 0, \\
       \Delta^{DN}(\la) \colon & \quad y(v_1) = y'(v_2) = 0, \\
       \Delta^{ND}(\la) \colon & \quad y'(v_1) = y(v_2) = 0, \\
       \Delta^{NN}(\la) \colon & \quad y'(v_1) = y'(v_2) = 0,
    \end{array}
$$
and with the conditions $BC$ in the vertices $v \in \partial G \backslash \{ v_1, v_2 \}$. Then
\begin{equation} \label{ND}
\Delta^{DD}(\la) \Delta^{NN}(\la) - \Delta^{DN}(\la) \Delta^{ND}(\la) \not \equiv 0.
\end{equation}
\end{lem}

\begin{proof}
We shall divide the proof into the following steps.
1. Let the tree $G$ consists of the only edge $[v_1, v_2]$. Then one can check the relation \eqref{ND} by direct calculation.

2. Let the vertices $v_1$ and $v_2$ be connected by edges with the same vertex $v$, and let there also be 
subtrees $G_i$, $i = \overline{1, n}$, from the vertex $v$ (see fig.~\ref{img:3}). 
Denote by $\Delta^D_i(\la)$ and $\Delta^N_i(\la)$ the characteristic functions for $G_i$ with 
the matching conditions \eqref{MC}, the boundary conditions $BC$ and $y(v) = 0$ for $\Delta^D_i(\la)$
and $y'(v) = 0$ for $\Delta^N_i(\la)$. According to Lemma~\ref{lem:detg}, the following relation holds
$$
 	\Delta^{DD}(\la) = \frac{\sin \rho T_1 \sin \rho T_2}{\rho^2} \Delta^K(\la) + \frac{1}{\rho}
 	(\sin \rho T_1 \cos \rho T_2 + \cos \rho T_1 \sin \rho T_2) \Delta^{\Pi}(\la), 
$$
where
$$
 	\Delta^{\Pi}(\la) = \prod_{i = 1}^n \Delta_i^D(\la), \quad \Delta^K(\la) = \Delta^{\Pi}(\la) \sum_{i = 1}^n \frac{\Delta_i^N(\la)}{\Delta_i^D(\la)}.
$$
Using similar representations for $\Delta^{NN}(\la)$, $\Delta^{DN}(\la)$ and $\Delta^{ND}(\la)$, we derive
$$
 \Delta^{DD}(\la) \Delta^{NN}(\la) - \Delta^{DN}(\la) \Delta^{ND}(\la) = -\left(\Delta^{\Pi}(\la)\right)^2 \not \equiv 0.	
$$

\begin{figure}[h!]
\centering
\begin{tikzpicture}
\filldraw (3, 3) circle (2pt) node[anchor=west]{$v$};
\filldraw (1, 3) circle (2pt) node[anchor=east]{$v_1$};
\filldraw (4.5, 4.5) circle (2pt) node[anchor=west]{$v_2$};
\draw[thick] (1, 3) edge node[above]{$e_1$} (3, 3);
\draw[thick] (4.5, 4.5) edge node[below]{$e_2$} (3, 3);
% G1
\draw[dashed] (3, 5) edge [bend left] (3, 3);
\draw[dashed] (1.5, 4.5) edge [bend right] (3, 3);
\draw[dashed] (3, 5) arc (20:200:0.8);
\filldraw (2, 5) circle (1pt);
\draw (2, 5) edge node[auto]{$G_1$}(3, 3);
% G2
\draw[dashed] (0, 1) edge [bend left] (3, 3);
\draw[dashed] (1, 0) edge [bend right] (3, 3);
\draw[dashed] (0, 1) arc (160:290:0.75);
\draw (3, 3) edge (1.5, 1.5);
\draw (1.5, 1.5) edge (0.5, 1);
\draw (1.5, 1.5) edge node[auto]{$G_2$} (1.5, 0.5);
\filldraw (1.5, 1.5) circle (1pt);
\filldraw (0.5, 1) circle (1pt);
\filldraw (1.5, 0.5) circle (1pt);
% G3
\draw[dashed] (6, 1) edge [bend right] (3, 3);
\draw[dashed] (5, 0) edge [bend left] (3, 3);
\draw[dashed] (5, 0) arc (-110:20:0.75);
\draw (3, 3) edge (4.5, 1.5);
\draw (4.5, 1.5) edge (5.5, 1);
\draw (4.5, 1.5) edge node[auto]{$G_3$} (4.5, 0.5);
\filldraw (4.5, 1.5) circle (1pt);
\filldraw (5.5, 1) circle (1pt);
\filldraw (4.5, 0.5) circle (1pt);
\end{tikzpicture}
\caption{}
\label{img:3}
\end{figure}
\begin{figure}[h!]
\centering
\begin{tikzpicture}
\filldraw (5, 8) circle (2pt) node[anchor=west]{$v_1$};
\filldraw (5, 6) circle (2pt) node[anchor=west]{$v_3$};
\filldraw (5, 0) circle (2pt) node[anchor=west]{$v_2$};
\filldraw (5, 2) circle (2pt) node[anchor=west]{$v_4$};
\draw[thick] (5, 8) edge node[right]{$e_1$} (5, 6);
\draw[thick] (5, 2) edge node[right]{$e_2$} (5, 0);
% G0
\filldraw (5, 4) circle (1pt);
\draw (5, 4) edge (5, 6);
\draw (5, 4) edge (5, 2);
\draw[dashed] (5, 4) circle (2);
\draw (5, 4) edge (6.5, 5);
\draw (5, 4) edge (6.5, 3);
\filldraw (6.5, 5) circle (1pt);
\filldraw (6.5, 3) circle (1pt);
\draw (5, 4) edge node[auto]{$G_0$}(4, 4);
\draw (4, 4) edge (3.5, 3);
\draw (4, 4) edge (3.5, 5);
\filldraw (4, 4) circle (1pt);
\filldraw (3.5, 3) circle (1pt);
\filldraw (3.5, 5) circle (1pt);
% G1
\draw[dashed] (5, 6) edge [bend left] (3.3, 7.3);
\draw[dashed] (5, 6) edge [bend right] (3.8, 7.8);
\draw[dashed] (3.3, 7.3) edge [bend left] (3.8, 7.8);
\draw (3.7, 7.3) node{$G_1$};
\draw (5, 6) edge (4, 7);
\filldraw (4, 7) circle (1pt);
% G2
\draw[dashed] (5, 6) edge [bend right] (6.7, 7.3);
\draw[dashed] (5, 6) edge [bend left] (6.2, 7.8);
\draw[dashed] (6.7, 7.3) edge [bend right] (6.2, 7.8);
\draw (6.3, 7.3) node{$G_2$};
\draw (5, 6) edge (6, 7);
\filldraw (6, 7) circle (1pt);
% tG1
\draw[dashed] (5, 2) edge [bend right] (3.3, 0.7);
\draw[dashed] (5, 2) edge [bend left] (3.8, 0.2);
\draw[dashed] (3.3, 0.7) edge [bend right] (3.8, 0.2);
\draw (3.7, 0.7) node{$\tilde G_1$};
\draw (5, 2) edge (4, 1);
\filldraw (4, 1) circle (1pt);
% tG2
\draw[dashed] (5, 2) edge [bend left] (6.7, 0.7);
\draw[dashed] (5, 2) edge [bend right] (6.2, 0.2);
\draw[dashed] (6.7, 0.7) edge [bend left] (6.2, 0.2);
\draw (6.3, 0.7) node{$\tilde G_2$};
\draw (5, 2) edge (6, 1);
\filldraw (6, 1) circle (1pt);
\end{tikzpicture}
\caption{}
\label{img:4}
\end{figure}

3. Now let the vertices $v_1$ and $v_2$ be connected by the edges with $v_3$ and $v_4$, respectively.
Let the tree $G$ splits by the vertices $v_3$ and $v_4$ into the subtrees $G_i$, $i = \overline{1, n_1}$, connected with $v_3$,
the subtrees $\tilde G_j$, $j = \overline{1, n_2}$, connected with $v_4$, the subtree $G_0$, including the both vertices $v_3$ and $v_4$,
and the edges $e_1$, $e_2$ (see fig.~\ref{img:4}). 
Denote by $\Delta^D_i(\la)$, $\Delta^N_i(\la)$, $i = \overline{1, n_1}$, and by 
$\tilde \Delta^D_j(\la)$, $\tilde \Delta^N_j(\la)$, $j = \overline{1, n_2}$, the characteristic functions for the subtrees
$G_i$ with the Dirichlet or Neumann boundary condition in $v_3$ and for the subtrees $\tilde G_i$
with the Dirichlet or Neumann boundary condition in $v_4$, respectively.
Let $\Delta_0^{DD}(\la)$, $\Delta_0^{DN}(\la)$, $\Delta_0^{ND}(\la)$ and $\Delta_0^{NN}(\la)$
be characteristic functions for the subtree $G_0$ with the following boundary conditions 
$$
    \begin{array}{ll}
       \Delta_0^{DD}(\la) \colon & \quad y(v_3) = y(v_4) = 0, \\
       \Delta_0^{DN}(\la) \colon & \quad y(v_3) = y'(v_4) = 0, \\
       \Delta_0^{ND}(\la) \colon & \quad y'(v_3) = y(v_4) = 0, \\
       \Delta_0^{NN}(\la) \colon & \quad y'(v_3) = y'(v_4) = 0,
    \end{array}
$$
and the conditions $BC$ in other boundary vertices.
Denote the functions
$$
 	\Delta_1^{\Pi}(\la) = \prod_{i = 1}^{n_1} \Delta_i^D(\la), \quad \Delta_2^{\Pi}(\la) = \prod_{j = 1}^{n_2} \tilde \Delta_j^D(\la), 
$$
$$
  	\Delta_1^K(\la) = \Delta_1^{\Pi}(\la) \sum_{i = 1}^n \frac{\Delta_i^N(\la)}{\Delta_i^D(\la)}, \quad
  	\Delta_2^K(\la) = \Delta_2^{\Pi}(\la) \sum_{j = 1}^n \frac{\tilde \Delta_j^N(\la)}{\tilde \Delta_j^D(\la)}.
$$
\begin{equation}   \label{defKP}
\left.
\begin{array}{ll}
 	\Delta^{KK}(\la) = & \Delta_0^{DD}(\la) \Delta_1^{K}(\la) \Delta_2^K(\la) + \Delta_0^{ND}(\la) \Delta_1^{\Pi}(\la) \Delta_2^{K}(\la) \\
 	 & + \Delta_0^{DN}(\la) \Delta_1^{K}(\la) \Delta_2^{\Pi}(\la) + \Delta_0^{NN}(\la) \Delta_1^{\Pi}(\la) \Delta_2^{\Pi}(\la), \\ 
  	\Delta^{\Pi K}(\la) = & \Delta_0^{DD}(\la) \Delta_1^{\Pi}(\la) \Delta_2^K(\la) + \Delta_0^{DN}(\la) \Delta_1^{\Pi}(\la) \Delta_2^{\Pi}(\la), \\
  	\Delta^{K \Pi}(\la) = & \Delta_0^{DD}(\la) \Delta_1^{K}(\la) \Delta_2^{\Pi}(\la) + \Delta_0^{ND}(\la) \Delta_1^{\Pi}(\la) \Delta_2^{\Pi}(\la),\\ 
  	\Delta^{\Pi \Pi}(\la) = & \Delta_0^{DD}(\la) \Delta_1^{\Pi}(\la) \Delta_2^{\Pi}(\la).
\end{array} \right\}
\end{equation}
In view of Lemma~\ref{lem:detg}, the following relation holds
\begin{multline*}
 	\Delta^{DD}(\la) = \frac{\sin \rho T_1 \sin \rho T_2}{\rho^2} \Delta^{KK}(\la) + 
 	\frac{\cos \rho T_1 \sin \rho T_2}{\rho} \Delta^{\Pi K}(\la) \\ +
 	\frac{\sin \rho T_1 \cos \rho T_2}{\rho} \Delta^{K \Pi}(\la) + 
 	\cos \rho T_1 \cos \rho T_2 \Delta^{\Pi \Pi}(\la).
\end{multline*}
Together with the similar relations for $\Delta^{DN}(\la)$, $\Delta^{ND}(\la)$ and $\Delta^{NN}(\la)$, it yields
$$
  	\Delta^{DD}(\la) \Delta^{NN}(\la) - \Delta^{DN}(\la) \Delta^{ND}(\la) = 
  	\Delta^{\Pi \Pi}(\la) \Delta^{KK}(\la) - \Delta^{\Pi K}(\la) \Delta^{K \Pi}(\la)
$$
Taking \eqref{defKP} into account, we obtain
$$
  	\Delta^{\Pi \Pi}(\la) \Delta^{KK}(\la) - \Delta^{\Pi K}(\la) \Delta^{K \Pi}(\la) = 
  	(\Delta_0^{DD}(\la) \Delta_0^{NN}(\la) - \Delta_0^{DN}(\la) \Delta_0^{ND}(\la)) \left( \Delta_1^{\Pi}(\la) \Delta_2^{\Pi}(\la) \right)^2.
$$
By virtue of Lemma~\ref{lem:asymptDelta}, $\Delta_i^{\Pi}(\la) \not \equiv 0$, $i = 1, 2$. Therefore the relation
\eqref{ND} holds for the tree $G$ if and only if it holds for the subtree $G_0$.
By induction, the claim of the lemma is valid for any tree $G$.
\end{proof}

By virtue of Lemmas~\ref{lem:asymptDelta}, \ref{lem:ND} and \eqref{formA}, $A(\la) \not \equiv 0$.
It follows from \eqref{formA}, \eqref{formB}, \eqref{formC}, that
\begin{multline*}
 	D(\la) = B^2(\la) - 4 A(\la) C(\la) \\ = F_1^2(\la) F_4^2(\la) \frac{\sin^2 \rho}{\rho^2} \Delta_0(\la) \left\{ \Delta_5^D(\la) \Pi(\la) + 
  	\Delta_5^D(\la) \frac{\sin \rho}{\rho} \xi(\la) + \Delta_4^{DD}(\la) \Delta_5^N(\la) \frac{\sin \rho}{\rho} \chi(\la)  \right\}^2.
\end{multline*}
Note that the expression in the bracket above equals to
$$
   \Delta_5^D(\la) \Pi(\la) + \frac{\sin \rho}{\rho} \Delta_0(\la). 
$$
Similarly to Lemma~\ref{lem:asymptDelta}, the following asymptotic formulas can be obtained:
$$
 	\Delta_5^D(\la) \Pi(\la) = C_1 r^{-p} \exp(r (T - 1))[1], \quad  
 	\frac{\sin \rho}{\rho} \Delta_0(\la) = C_2 r^{-q} \exp(r (T + 1))[1], 
$$
where $\rho = i r$, $r \to +\iy$, $T = \sum\limits_{j = 1}^m T_j$,
$C_1$, $C_2$, $p$ and $q$ are some constants. Clearly, the second term grows faster than the first one.
Therefore $\Delta_0(\la) \not \equiv 0$ implies $D(\la) \not \equiv 0$. The proof of Lemma~\ref{lem:discr} is finished.

Using Lemma~\ref{lem:asymptDelta}, one can also check, that $B(\la)$ and $\sqrt{D(\la)}$ have the same power of $\rho$ 
in the denominator, so the roots of \eqref{quad} have different asymptotic behavior.

\bigskip

{\large \bf 5. Example}

\medskip

In this section, we provide the solution of Inverse Problem~\ref{ip:main}
for the example of the graph in the fig.~\ref{img:2}. For simplicity, let $T_j = 1$, $j = \overline{1, 5}$.
 Let $x_3 = 0$ corresponds to the vertex $v_3$ and
$x_3 = 1$ corresponds to $v_6$. For the boundary edges, $x_j = 0$ correspond to the boundary vertices.
The matching conditions \eqref{MC} take the form
\begin{equation} \label{MCex}
\begin{array}{l}
v_3 \colon \quad y_1(1) = y_2(1) = y_3(0), \quad y'_1(1) + y'_2(1) - y'_3(0) = 0, \\
v_6 \colon \quad y_3(1) = y_4(1) = y_5(1), \quad y'_3(1) + y'_4(1) + y'_5(1) = 0.
\end{array}
\end{equation}

\begin{figure}[h!]
\centering
\begin{tikzpicture}
\filldraw (2, 2) circle (2pt) node[anchor=west]{$v_6$};
\filldraw (2, 4) circle (2pt) node[anchor=west]{$v_3$};
\filldraw (0, 0) circle (2pt) node[anchor=north]{$v_4$};
\filldraw (4, 0) circle (2pt) node[anchor=north]{$v_5$};
\filldraw (0, 6) circle (2pt) node[anchor=south]{$v_1$};
\filldraw (4, 6) circle (2pt) node[anchor=south]{$v_2$};
\draw[thick] (2, 2) edge node[auto]{$e_3$} (2, 4);
\draw[thick] (2, 2) edge node[below]{$e_4$} (0, 0);
\draw[thick] (2, 2) edge node[below]{$e_5$} (4, 0);
\draw[thick] (2, 4) edge node[above]{$e_1$} (0, 6);
\draw[thick] (2, 4) edge node[above]{$e_2$} (4, 6);
\end{tikzpicture}
\caption{Example}
\label{img:2}
\end{figure}

For this example, each subtree $G_i$ consists of only one edge $e_i$, $i = \overline{1, 5}$.
Let us know the spectra $\Lambda_0$, $\Lambda_1$, $\Lambda_4$ and the potential $q_3$. Using the given spectra,
one can easily find the characteristic functions $\Delta_0(\la)$, $\Delta_1(\la)$, $\Delta_4(\la)$ and
the Weyl functions $M_1(\la)$, $M_4(\la)$. Solving problems IP(1) and IP(4), recover $q_1$ and $q_3$.  

Consider the boundary value problem $L$. Represent the solution $y$ in the form \eqref{expy} and substitute
it into \eqref{MC} and \eqref{boundL}. From \eqref{boundL}, one gets
$M_1^0(\la) = M_2^0(\la) = M_4^0(\la) = M_5^0(\la) = 0$. Then matching conditions \eqref{MCex}
yield the system
\begin{equation} \label{systemex}
 	\begin{pmatrix}
 	   S_1 & -S_2 & 0 & 0 & 0 & 0 \\
 	   0 & S_2 & -1 & 0 & 0 & 0 \\
 	   S'_1 & S'_2 & 0 & -1 & 0 & 0 \\
 	   0 & 0 & C_3 & S_3 & -S_4 & 0 \\
 	   0 & 0 & 0 & 0 & S_4 & -S_5 \\
 	   0 & 0 & C'_3 & S'_3 & S'_4 & S'_5
 	\end{pmatrix}
 	\begin{pmatrix}
 	   M_1^1 \\ M_2^1 \\ M_3^0 \\ M_3^1 \\ M_4^1 \\ M_5^1
 	\end{pmatrix}
 	= 0.
\end{equation}
Here we omit arguments $(1, \la)$ and $(\la)$ for brevity.
The characteristic function $\Delta_0(\la)$ equals the determinant of \eqref{systemex}. 
Since we know $q_1$, $q_3$ and $q_4$, we can solve \eqref{eqv} and obtain the functions $S_j(x_j, \la)$ 
and $C_j(x_j, \la)$ for $j = 1, 3, 4$. Therefore the determinant admits the following representation
$$
 	\Delta_0 = a_{11} S_2 S_5 + a_{12} S'_2 S_5 + a_{13} S_2 S'_5 + a_{14} S'_2 S'_5,
$$ 
where
$$
	\begin{array}{l}
 	a_{11} = S'_1 \begin{vmatrix} S_3 & -S_4 \\ S'_3 & S'_4 \end{vmatrix} + 
 	S_1 \begin{vmatrix} C_3 & -S_4 \\ C'_3 & S'_4 \end{vmatrix}, \quad
 	a_{12} = S_1 \begin{vmatrix} S_3 & -S_4 \\ S'_3 & S'_4 \end{vmatrix}, \\
 	a_{13} =  (S'_1 S_3 + S_1 C_3) S_4, \quad a_{14} = S_1 S_3 S_4.
 	\end{array}
$$
If one change $S_1$ to $C_1$ or $S_4$ to $C_4$, he obtains analogous relations for $\Delta_1(\la)$ and $\Delta_4(\la)$,
respectively. Thus we arrive at the system \eqref{systemA}. 

Let $q \equiv 0$ on $G$. Then
$$
 	C^0_j(x_j, \la) = \cos \rho x_j, \quad S^0_j(x_j, \la) = \frac{\sin \rho x_j}{\rho},
$$
$$
 	a^0_{11} = \frac{\sin 3 \rho}{\rho}, \quad a^0_{12} = a^0_{13} = \frac{\sin 2 \rho \sin \rho}{\rho^2}, \quad
 	a^0_{14} = \frac{\sin^3 \rho}{\rho^3},
$$
$$
 	a^0_{21} = a^0_{31} = \cos 3 \rho, \quad a^0_{22} = a^0_{33} = \frac{\sin 2 \rho \cos \rho}{\rho}, \quad
$$
$$
 	a^0_{23} = a^0_{32} = \frac{\cos 2 \rho \sin \rho}{\rho}, \quad a^0_{24} = a^0_{34} = \frac{\cos \rho \sin^2 \rho}{\rho^2}.
$$
$$
 	\Delta^0_0 = \frac{- 9 \sin 5 \rho + 13 \sin 3 \rho + 6 \sin \rho}{16 \rho^3}, \quad
 	\Delta^0_1 = \Delta^0_4 = \frac{- 9 \cos 5 \rho + 7 \cos 3 \rho + 2 \cos \rho}{16 \rho^2}.
$$
Using \eqref{defb}, we obtain
$$
 	b^0_{11} = b^0_{21} = \frac{-3 \sin 6 \rho - 2 \sin 4 \rho + 13 \sin 2 \rho}{16 \rho^3}, \quad
 	b^0_{12} = b^0_{23} = \frac{- 3 \cos 6 \rho + 6 \cos 4 \rho + 3 \cos 2 \rho - 6}{16 \rho^4},
$$
$$
	b^0_{13} = b^0_{22} = \frac{3 \cos 6 \rho - 10 \cos 4 \rho + 13 \cos 2 \rho - 6}{32 \rho^4}, \quad
	b^0_{14} = b^0_{24} = \frac{- 3 \sin 6 \rho + 12 \sin 4 \rho - 15 \sin 2 \rho}{32 \rho^5}.
$$
Substitute these formulas into \eqref{defABC}:
$$
 	A_0 = \frac{-27 \sin 12 \rho + 174 \sin 10 \rho - 420 \sin 8 \rho + 378 \sin 6 \rho + 153 \sin 4 \rho - 468 \sin 2 \rho}{2048 \rho^9}, 
$$
$$
 	B_0 = \frac{-27 \cos 12 \rho + 84 \cos 10 \rho + 106 \cos 8 \rho - 764 \cos 6 \rho + 1099 \cos 4 \rho -344 \cos 2 \rho -154}{2048 \rho^8},
$$
$$
  	C_0 = \frac{-27 \sin 12 \rho + 48 \sin 10 \rho + 140 \sin 8 \rho - 336 \sin 6 \rho -71 \sin 4 \rho +512 \sin 2 \rho}{1024 \rho^7}.
$$
Calculate the discriminant of equation \eqref{quad0}:
\begin{multline*} 
 	D_0 = B_0^2 - 4 A_0 C_0 = (6561 \cos 24 \rho - 52488 \cos 22 \rho + 128628 \cos 20 \rho + 83592 \cos 18 \rho \\ -
 	987134 \cos 16 \rho + 1543976 \cos 14 \rho + 702372 \cos 12 \rho -4646312 \cos 10 \rho + 3755087 \cos 8 \rho \\ +
 	3053616 \cos 6 \rho - 4805144 \cos 4 \rho - 4176688 \cos 2 \rho + 5393934) / (8388608 \rho^{16}). 
\end{multline*}
We used wxMaxima 12.04.0 for calculations.

Obviously, $A_0(\la) \ne 0$, $D_0(\la) \ne 0$, so according to Lemma~\ref{lem:asymptABC}, the roots of equation \eqref{quad}
in the general case have different asymptotics:
$$
 	\tilde M_2^1(\la) = \frac{\rho \cos \rho}{\sin \rho}[1], \quad \tilde M_2^2(\la) = -\frac{1 + 6 \cos^2 \rho}{3 \sin \rho \cos \rho}[1].
$$ 
Since $\tilde M_2(\la) = \frac{S'_2(1, \la)}{S_2(1, \la)}$, only the root $M_2^1(\la)$ is the required one.

Finally, one can easily find $\tilde M_5(\la)$ and solve classical Sturm-Liouville inverse problems by Weyl functions
on the edges $e_2$ and $e_5$.

Now let us consider the case when the potential is known a priori on two edges.
If they are $e_1$ and $e_4$, then only two spectra $\Lambda_0$ and $\Lambda_2$ are sufficient to recover the potential on the whole graph. Indeed, one can solve IP(2), then apply Theorem~\ref{thm:cut} to the vertex $v_3$, find $q_3$ and then similarly find $q_5$. However, the knowledge of $q_1$ and $q_2$ do not allow us to recover the potential from two spectra by our method.  If we have only $\Lambda_0$ and $\Lambda_4$, we can not recover $q_3$. Similarly, if we know $q_3$ initially, the knowledge of the potential on one of the boundary edges do not allow us to reduce the number of given spectra. Thus, if the potential is known on multiple edges, the number of required spectra depends on the location of these edges.

\medskip

{\bf Acknowledgments}. This work was supported by Grant 1.1436.2014K
of the Russian Ministry of Education and Science, by Grants 15-01-04864 and 16-01-00015
of Russian Foundation for Basic Research and by the Mathematics Research Promotion Center of Taiwan.

\vspace{1cm}

Natalia Bondarenko

Department of Applied Mathematics

Samara University

Moskovskoye sh. 34, Samara 443086, Russia

Department of Mechanics and Mathematics

Saratov State University

Astrakhanskaya 83, Saratov 410012, Russia

{\it bondarenkonp@info.sgu.ru}

\medskip

Chung-Tsun Shieh

Department of Mathematics

Tamkang University

151 Ying-chuan Road Tamsui, Taipei County, Taiwan 25137, R.O.C.

{\it ctshieh@mail.tku.edu.tw}  

\end{document}